\documentclass[12pt]{amsart}
	\usepackage{amsmath}
	\usepackage{amsfonts}
	\usepackage{amssymb}
	\usepackage[]{amsmath, amsthm, amsfonts, amssymb, epsfig}
	\newtheorem{thm}{Theorem}
\begin{document}
	
		\author[A. R. Legg]{A. R. Legg}
		\address{Department of Mathematical Sciences, Purdue University Fort Wayne, Ft. Wayne, IN 46805}
		\email{leggar01@pfw.edu}
%
	
		\author{P. D. Dragnev$^\dagger$}
		\address{Department of Mathematical Sciences, Purdue University Fort Wayne, Ft. Wayne, IN 46805}
		\email{dragnevp@pfw.edu}
		\thanks{\noindent $^\dagger$ The research of this author was supported in part by a Simons Foundation CGM no. 282207.}

		\begin{abstract}
			With the sphere $\mathbb{S}^2 \subset \mathbb{R}^3$ as a conductor holding a unit charge with logarithmic interactions, we consider the problem of determining the support of the equilibrium measure in the presence of an external field consisting of finitely many point charges on the surface of the sphere.  We determine that for any such configuration, the complement of the equilibrium support is the stereographic preimage from the plane of a union of classical quadrature domains,  whose orders sum to the number of point charges.
		\end{abstract}
		
		\title{Logarithmic Equilibrium on the Sphere in the Presence of Multiple Point Charges}
		
		\maketitle
				

\vskip 2mm

\vskip 2mm

	\section{Introduction to the problem}\label{Sect1}
		
Consider the unit sphere $\mathbb{S}^2 \subset \mathbb{R}^3$ as a conductor, carrying a unit positive electric charge which is free to distribute into the Borel measure which will uniquely minimize logarithmic energy. With no other external field present, we of course intuit that the equilibrium state is uniform over the whole sphere.  But what happens in the presence of an added field?  

 The case of an external field consisting of a single point charge has been considered in \cite{Dr}, with the conclusion that the equilibrium support is the complement of a perfect spherical cap centered at the point charge. That is to say, a single point charge tends to repel the charge on the sphere, so that a perfect cap is swept clean of charge.  The radius of the cap can be explicitly calculated based on the intensity of the point charge, and the result can be extended to Riesz energies of various exponent, and even to higher dimensions (see \cite{DS} and \cite{BDS}).
 
 In \cite{BDSW}, the case of multiple point charges is undertaken, and the authors demonstrate that, similar to the single-point-charge case, the equilibrium support is the complement of the union of spherical caps centered at the various point charges, with the caveat that this holds only in case the interiors of these ``caps of influence" do not overlap.  Numerically generated graphics are shown there that illustrate the case when two point charges' caps of influence do overlap, and what arises is an apparently smooth lobe-shaped equilibrium support excluding both of the individual caps of influence. 
 
 The open question raised there, then, is how exactly to characterize the equilibrium support when multiple charges are present, and the charges are close enough or strong enough that their individual caps of influence overlap.
 
We do so by means of classical planar quadrature domains.  By projecting stereographically to the plane and then pursuing a course of complex analysis and potential theory, we show that the region of charge exclusion is the stereographic preimage from the plane of a quadrature domain, the order of whose quadrature identity is equal to the number of point charges constituting the external field. Indeed, this will hold for any finite number of point charges in any configuration on the sphere.

 This sheds more light on the \cite{BDSW} result, since in the case of one point charge the only possible quadrature domain of order one is known to be the disc, and the stereographic preimage of a disc onto the sphere is a spherical cap. 
 
 The particular case of two charges of equal intensity whose regions of influence overlap was recently studied by Criado del Rey and Kuijlaars \cite{CdRK}, with methods quite different than our own.

 To begin, we will review some notions from potential theory, complex analysis, and quadrature domain theory which will be encountered in our explication. For a more thorough introduction to logarithmic potentials in the plane, we point the reader to \cite{SaffTotik}.
 
 After the review, we approach the problem from several perspectives. First, a general connection to balayage is made, which will reinforce the theme from \cite{BDSW} that as long as the components of the complement of the equilibrium support are disjoint, they are determined separately from one another (so groups of charges really do have proper ``regions of influence".)  
 
 Next we show that, assuming a priori smoothness of the boundary of the equilibrium support, Frostman's condition on the equilibrium potential can be used with Mergelyan's Theorem to identify the complement of the equilibrium support as a quadrature domain as described above.
 
 In Section \ref{Sect5}, we offer an alternate approach which assumes no a priori boundary smoothness whatever. This approach mirrors the development of Aharanov and Shapiro in \cite{AS}, modified to our present problem. 
 
 Finally, we present some examples that illustrate our results.

\section{Potential Theory Background}
\label{Sect2}
\subsection{Equilibrium Measures and External Fields}

The mathematical presentation of our problem is as follows: Let $a_i,$ $i=1,2, \cdots, n$ be $n$ distinct points on the unit sphere $\mathbb{S}^2 \subset \mathbb{R}^3$, and for points $x \in \mathbb{S}^2$, consider the collection of point charges

\[ \mu_n:=\sum_{i=1}^n q_i \delta_{a_i},\]
where $q_i$ are positive real numbers, and the $\delta$'s are Dirac point distributions. Assuming logarithmic interaction, we will consider the external field produced by the charges, expressed at points $x$ as:
  \begin{equation}\label{Q} 
  Q(x) = \sum_{i=1}^n q_i\log \frac{1}{|x-a_i|} .
  \end{equation}
 
  Let $\mathcal{M}$ denote the set of all unit Borel measures on $\mathbb{S}^2$.
Then, given a $\mu \in \mathcal{M}$, the {\em logarithmic potential} of $\mu$ at the point $x \in \mathbb{S}^2$ is
  \[U^{\mu}(x)=\int_{\mathbb{S}^2} \ln \frac{1}{|y-x|} d\mu(y);\] 
    and the logarithmic energy of $\mu$ is 
    \[V^\mu= \int_{\mathbb{S}^2} U^{\mu}(x) d\mu(x).\]
    Given a compact subset $K\subset \mathbb{S}^2$ and measures supported on $K$ with finite energy, one can seek for a minimizer of $V^\mu$ among the class $\mathcal{M}(K)$ of probability measures supported on $K$. Such a minimizer $\mu_K$, referred to as the {\em equilibrium measure} of $K$, exists and is unique. The logarithmic capacity of $K$ is defined as $\text{cap}(K) = \exp(-V^{\mu_K}).$

    In the presence of an {\em external field} $Q$, the total weighted energy of the system is
     
      \[V_Q^\mu=\int_{\mathbb{S}^2} \int_{\mathbb{S}^2} \ln \frac{1}{|y-x|} d\mu(y) d\mu(x) + 2\int_{\mathbb{S}^2} Q(y) d\mu(y).\]
      The {\em equilibrium measure w.r.t. to the external field} $\mu_Q$ is then defined as the unique unit Borel measure with minimal possible weighted energy. 
      
     In Section 3 we shall consider minimal energy problems over the class $\mathcal{M}^t$ of measures of total mass $t>0$. The extremal minimizers in this case are denoted with $\mu_K^t$ and $\mu_Q^t$. We note the modification of the weighted energy
          \begin{equation}
          \label{t_eq}
          V_{Q}^{\mu,t}=\int_{\mathbb{S}^2} \int_{\mathbb{S}^2} \ln \frac{1}{|y-x|} d\mu(y) d\mu(x) + 2t\int_{\mathbb{S}^2} Q(y) d\mu(y).
          \end{equation}

Setting aside the sphere for a moment, in the planar setting the equilibrium measure under the influence of an admissible external field $Q$ on a conductor $\Omega$ of positive logarithmic capacity is described by the so-called Frostman Theorem, which we include here for reference. For an explanation and proofs, see \cite{SaffTotik}.

\begin{thm} \label{thm1} Let $\Omega \subset \mathbb{C}$ have positive logarithmic capacity, and let $Q$ be an admissible external field. Then consider the problem of minimizing the weighted energy $V_Q$ among all positive unit Borel measures with compact support in $\overline{\Omega}$. The following hold:
	
	(i)  The minimal energy $V_Q$ is finite and obtained by a unique minimizing measure $\mu_Q$ (called the equilibrium measure).
	
	(ii)  For some constant $F_Q$, $U^{\mu_Q}(z)+Q(z)\leq F_Q$ on $supp(\mu_Q)$, and $U^{\mu_Q}(z)+Q(z) \geq F_Q$ quasi-everywhere on $\Omega$  (i.e. with the exception of a set of zero logarithmic capacity).
	
	(iii) The measure $\mu_Q$ is uniquely characterized by (ii).
\end{thm}

We note that for the external fields considered in this article (see \eqref{Q}) the second inequality in (ii) holds everywhere and therefore the weighted potential of $\mu_Q$ is constant on the support of $\mu_Q$, and can only be greater or equal outside of the support. The same characterization holds for $\mu_{Q}^t$ as well.

As for the sphere, in \cite{BDSW} it is shown that this problem can be considered in the plane via stereographic projection, using the fact that under this projection, taking the north pole as the source of the projection, surface area measure on the sphere transforms to become
\[ \frac{dA}{\pi (1+|z|^2)^2} \]
 on the complex plane (here and throughout, $dA$ shall refer to planar Lebesgue measure). In our case of a $Q$ defined in \eqref{Q}, the resulting projected planar formulation of Frostman's condition $(ii)$ is:

\begin{equation}
\begin{split}
\label{Frostman}
 \frac{q+1}{\pi}\int_{\Sigma^*} \frac{1}{(1+|w|^2)^2}\ln\frac{1}{|w-z|}dA_w + \sum_{i=1}^n q_i \ln \frac{1}{|z-z_i|} +(q+1)\ln \sqrt{1+|z|^2}\\=const,
 \end{split}
\end{equation}
where the stereographic projections of the point charges are at points $z_i\in \mathbb{C}$, with respective charge intensities of $q_i$, and the total sum of all the charges is $q=\sum_{i=1}^n q_i.$ The equilibrium support on the sphere is $\Sigma$, and its stereographic projection onto the plane is called $\Sigma^*$. The equality is valid for all $z \in \Sigma^*$. 

Our purpose throughout this article, then, is to identify for which $\Sigma^*$ this equality could possibly hold.

An important note here is that, as can readily be checked, if the north pole for the projection is taken to be at one of the point charges, the resulting external field in the plane is admissible. This allows for the utilization of Theorem \ref{thm1} to derive the analogous theorem of existence, uniqueness, and characterization of the equilibrium measure $\mu_Q$ on $\mathbb{S}^2$ with external field $Q(x)$ defined in \eqref{Q}. If the projection is taken from any other point, the external field is instead `weakly admissible' in the sense described by Bloom, Levenberg and Wielonsky \cite{BLW}. Thankfully, the Frostman condition remains intact. This means that rotations of the sphere result in no loss of generality in using (\ref{Frostman}) to describe the equilibrium support.

\subsection{Quadrature Domains}

Our characterization of the equilibrium support will involve planar quadrature domains, which generalize the harmonic mean value property of discs. A domain $\Omega \subset \mathbb{C}$ is a {\em quadrature domain} for integrable analytic functions if there exist finitely many points $z_i \in \Omega$ and constants $c_{ik}$ such that for any integrable analytic $f$ on $\Omega$, we have the {\em quadrature identity}
\[\int_\Omega f(w)dA_w=\sum_{i,k} c_{ik}f^{(k)}(z_i).\]

In other words, integration for such $f$ is identical to a finite linear combination of point evaluations of the functions and their derivatives, and the same coefficients and points apply for each $f$. The `order' of a quadrature domain is the number of terms in its quadrature identity, and the points of evaluation are called `nodes'.

The theory of quadrature domains has gained attention from several areas in the past few decades, in no small part because they automatically enjoy a long list of desirable properties, and exist in abundance. Their first manifestations occur in \cite{AS}, and from there they are applied to such fields as fluid dynamics, operator theory, real potential theory, and complex analysis. To name just a few references which give a good background, we suggest \cite{Bell4,GS,Shapiro2, Proc}, and their respective bibliographies. Connections between quadrature domains and the sphere appear from the realm of fluid dynamics in \cite{Crowdy, CC} and in the treatment of potential theory on manifolds in \cite{Skinner, GR}. 

Quadrature domains can be generalized by changing the test class of functions on which the quadrature identity holds, or by replacing the sum of point evaluations by compactly supported measures in the domain. We will not do so here, and will employ only the `classical' quadrature domains which we defined above.  

Among the many approaches to thinking about quadrature domains, our analysis will specifically apply the concept of a `Schwarz function'.  Given a bounded  domain $\Omega$ in the complex plane, the Schwarz function $S(w)$ of the boundary of $\Omega$, if it exists, is defined as the analytic continuation of the function
 \[\bar{w}|_{bd \Omega}\]
  into some interior neighborhood. If $\Omega$ is real analytic, the existence of the Schwarz function is guaranteed at least in a small neighborhood of $bd\Omega$ by the Cauchy-Kovalevskaya theorem.  If the Schwarz function extends inside a bounded domain $\Omega$ in such a way as to be meromorphic throughout, with finitely many poles, then $\Omega$ turns out to be a quadrature domain. This can be conceptualized as Stokes's theorem paired with the Residue theorem, since for an analytic $f$, \[f dw \wedge  d\bar{w}=d(\bar{w} f dw), \]
   and on the boundary $\bar{w}=S(w)$, which is meromorphic. We can see from this also that the number of point evaluations in the quadrature identity is equal to the number of poles of $S(w)$ counting multiplicity. We review this well-known information here:

\begin{thm}
	Let $\Omega$ be a bounded domain in the plane. If the boundary function \[\bar{w}|_{bd \Omega}\]
	extends to be a meromorphic function $S(w)$ (called the Schwarz function) on $\Omega$ with finitely many poles, then $\Omega$ is a quadrature domain whose order is the number of poles of $S(w)$ counted with multiplicity.
\end{thm}

\section{Equilibrium Measures via Balayage and Signed Equilibria}
\label{Sect3}

In this section we shall introduce the notion of a {\em {\rm (}logarithmic{\rm )}  balayage} of a measure and utilize it to characterize the equilibrium measure $\mu_Q$ (see \cite{La, SaffTotik}). Given a positive measure $\nu$ on the unit sphere $\mathbb{S}^2$, its balayage $\widehat{\nu}:={\rm Bal}(\nu, K)$ on a compact subset $K \subset \mathbb{S}^2$ is defined as the unique measure supported on $K$, that preserves, up to a constant, the logarithmic potential of $\nu$ on $K$, and diminishes it on the whole sphere, namely
\begin{equation}
U^{\widehat{\nu}}(x)=U^\nu (x)+c \quad {\rm on} \quad K,\quad \quad U^{\widehat{\nu}}(x)\leq U^\nu (x)+c \quad {\rm on} \quad \mathbb{S}^2
\end{equation}
 We note that balayage preserves the total mass, that is $\nu(\mathbb{S}^2)=\widehat{\nu}(\mathbb{S}^2)$. 
 
 There are various techniques for finding balayage of measures. For example, balayage may be found in steps. Say, $F\subset K \subset \mathbb{S}^2$, then ${\rm Bal}(\nu,F)={\rm Bal} ({\rm Bal}(\nu,K),F)$. To find the balayage of a point-mass measure $\delta_a$ at a point $a \in K^c:=\mathbb{S}^2\setminus K$, we invert (perform a stereo-graphical projection) the sphere about $a$ and determine the equilibrium measure $\mu_{K^*}$ of the image $K^*$ of $K$. The pre-image of $\mu_{K^*}$ under the stereo-graphical projection is the balayage $\widehat{\delta_a}$. The following superposition formula is also useful
 \[ {\rm Bal}(\nu,K)=\nu_{\vert K} + {\rm Bal}(\nu_{\vert K^c} ,K)=\nu_{\vert K} +\int_{K^c} \widehat{\delta_y}\, d\nu(y).\]

This allows us to make an important observation about where the logarithmic balayage on the sphere "lives". Should we fix the point $a$ at the north pole and use $\sqrt{2}$ as the inversion radius, the image of the sphere is $\mathbb{C}$. It is known that the equilibrium measure $\mu_{K^*}$ is supported on the outer boundary of $K^*$, which yields that $\widehat{\delta_a}$ is supported on the boundary of the component of $K^c=\mathbb{S}^2\setminus K$ that includes $\delta_a$, i.e. $a$ will not have "electrostatic influence" on the other components of $K^c$. The superposition formula extends this conclusion to ${\rm Bal}(\nu_{\vert K^c} ,K)$.
 
We are now in a position to extend the result from \cite{BDSW} that disjoint components of the complement of the equilibrium support ${\rm supp}(\mu_Q)$ are determined independently from each other. In this regard, we remind the reader that the characterization in Theorem \ref{thm1} holds for $\mu_{Q}^t$, namely
\begin{equation}\label{SphVarIneq} U^{\mu_Q^t}(x)+Q(x)\geq F_{Q,t} \ {\rm on }\ \mathbb{S}^2,\quad U^{\mu_Q^t}(x)+Q(x) = F_{Q,t} \ {\rm on} \ {\rm supp}(\mu_Q^t) .
\end{equation}

\begin{thm} \label{bal} Let $Q(x)$ be a discrete external field on the unit sphere $\mathbb{S}^2$  given in \eqref{Q} and let $\Sigma=\Sigma_Q$ be the support of the (unique) equilibrium measure $\mu_Q$.  Denote the connected components of $\Sigma^c$ with $C_1,\dots,C_m$ and define the associated with these components discrete measures $\mu_{n,j}$ and the related to these measures external fields $Q_{j}$ 
\[ \mu_{n,j}:= \sum_{a_i \in C_j} q_i \delta_{a_i},\quad Q_j (x):= U^{\mu_{n,j}}(x)= \sum_{a_i \in C_j} q_i\log \frac{1}{|x-a_i|},\]
i.e. $Q(x)=Q_1 (x)+\dots +Q_m(x)$. Then the components $\{ C_j\}$ are determined uniquely by the condition that
\begin{equation}\label{BalCond}
{\rm Bal}((1+q){\sigma}_{\vert_{\Sigma^c} } , \Sigma )-\sum_{j=1}^m {\rm Bal}\left( \mu_{n,j}, \partial C_j \right)\equiv 0,
\end{equation}
where $\sigma$ is the normalized unit Lebesgue surface measure on $\mathbb{S}^2$.
Consequently, $\mu_Q = (1+q) \sigma_{\vert \Sigma}$. 

Furthermore, for every $j=1,\dots,m$, the equilibrium measures with respect to $Q_j$ of norm $t_j :=1+q-\|\mu_{n,j}\|$ are given as
\begin{equation} \label{mu_Q^t}
\mu_{Q_j}^{t_j}= (1+q) \sigma_{\vert C_j^c},\end{equation}
determined uniquely by the condition
\begin{equation}\label{mu_Q^tcond}  {\rm Bal}((1+q){\sigma}_{\vert_{C_j} } , C_j^c )-{\rm Bal}\left( \mu_{n,j}, \partial C_j \right)\equiv 0.\end{equation}
\end{thm} 

{\bf Remark:} Note that as discussed above, for logarithmic interaction potentials we have ${\rm Bal}((1+q){\sigma_2}_{\vert_{C_j} } , C_j^c )={\rm Bal}((1+q){\sigma_2}_{\vert_{C_j} } , \partial C_j )$. 

 \begin{proof}  We first describe the conversion of logarithmic potentials on the sphere and the complex plane under stereographic projection. Let $a\in \mathbb{S}^2$ and let $z,w \in \mathbb{C}$ be the stereographic images of $x,y\in \mathbb{S}^2$ respectively under inversion centered at $a$ with radius $\sqrt{2}$, i.e. $|x-a| \cdot |z-a|=|y-a|\cdot |w-a|=2$. Let the measure in the complex plane $\mu^*$ be the image of a measure $\mu$ supported on $\mathbb{S}^2$. It is clear that $\mu^* (\mathbb{C})=\mu(\mathbb{S}^2)$. Utilizing the distance conversion formula
 \[ |x-y|=\frac{2|z-w|}{|z-a||w-a|}, \] 
 we derive the following {\em spherical-to-complex potentials formula}
 \begin{equation} \label{Sphere-to-Complex} U^\mu (x)=U^{\mu^*}(z)+\|\mu^*\|\log\frac{|z-a|}{2}-U^{\mu^*}(a).\end{equation}
 
 Next, we shall find a balayage representation of the equilibrium measure $\mu_Q$. Denote the signed measure
 \[\eta:={\rm Bal}((1+q)\sigma,\Sigma)-{\rm Bal}(\mu_n,\Sigma).  \]
 Clearly, $\eta$ is supported on $\Sigma$ and its weighted potential satisfies
 \[ U^\eta (x)+Q(x)=(1+q)U^\sigma (x)=(1+q)V^\sigma, \quad x\in \Sigma. \]
 On the other hand, from the spherical counterpart of Theorem 1 
 \[ U^{\mu_Q}(x)+Q(x)=F_Q,\quad x\in {\rm supp}(\mu_Q)=\Sigma.\]
 This implies that $U^{\mu_Q-\eta}(x)=const$ on $\Sigma$, and hence $V^{\mu_Q-\eta}=0$ as the total mass of the signed measure $\mu_Q-\eta$ is zero. Since both measures, $\mu_Q$ and $|\eta|=\eta^++\eta^-$ have finite logarithmic energies, using  \cite[Theorem 4.1]{Si} one concludes that $\mu_Q-\eta\equiv 0$ and the balayage representation
 \begin{equation}\label{BalRepr}
 \mu_Q={\rm Bal}((1+q)\sigma,\Sigma)-{\rm Bal}(\mu_n,\Sigma)
 \end{equation}
 holds. Observe that 
 \[ {\rm Bal}((1+q)\sigma,\Sigma)=(1+q) \sigma_{\vert \Sigma}+\sum_{j=1}^m {\rm Bal}((1+q)\sigma_{\vert C_j},\Sigma)\]
 and 
 \[ {\rm Bal}(\mu_n,\Sigma)= \sum_{j=1}^m {\rm Bal}(\mu_{n,j},\Sigma).\]
 Utilizing the fact that for every $j=1,\dots,m$
 \[{\rm Bal}((1+q)\sigma_{\vert C_j},\Sigma) ={\rm Bal}((1+q)\sigma_{\vert C_j},\partial C_j), \ {\rm Bal}(\mu_{n,j},\Sigma)={\rm Bal}(\mu_{n,j},\partial C_j)\] 
we can further expand \eqref{BalRepr} as
 \[
 \mu_Q=(1+q) \sigma_{\vert \Sigma}+\sum_{j=1}^m \left[ {\rm Bal}((1+q)\sigma_{\vert C_j},\partial C_j)-{\rm Bal}(\mu_{n,j},\partial C_j)\right],
 \]
 which implies that
 \begin{equation} \label{nu_j}
 \nu_j := (1+q)\sigma_{\vert \partial C_j} +{\rm Bal}((1+q)\sigma_{\vert C_j},\partial C_j)-{\rm Bal}(\mu_{n,j},\partial C_j) \geq 0,\ j=1,\dots,m .\end{equation}
 
Next, let us consider $\epsilon>0$ small enough, so that the set $\Sigma_\epsilon$ obtained by removing from $\mathbb{S}^2$ open disjoint spherical caps of radius $\epsilon$ with centers $a_j$ includes in its interior $\Sigma$. This is possible as ${\rm supp}(\mu_Q)$ is contained in a set $\{ x: Q(x)\leq C\}$ for some $C$ large enough. Consider the {\em signed equilibrium} on $\Sigma_\epsilon$ associated with $Q$, namely the unique signed measure $\eta_\epsilon$, such that  $\eta_\epsilon(\mathbb{S}^2)=1$ and
\[ U^{\eta_\epsilon}(x)+Q(x)=F_\epsilon, \ x\in \Sigma_\epsilon  \]
for some constant $F_\epsilon$ (see \cite{BDS} for details).
The signed equilibrium was found in \cite{BDSW} as
\[ \eta_\epsilon = (1+q) \sigma_{\vert \Sigma_\epsilon}+ {\rm Bal}((1+q)\sigma_{\vert \Sigma_\epsilon^c},\Sigma_\epsilon)-{\rm Bal}(\mu_n,\Sigma_\epsilon) .\]
Utilizing \eqref{SphVarIneq} for $t=1$ we derive
\[U^{\mu_Q+{\rm Bal}(\mu_n,\Sigma_\epsilon)}(x)\geq U^{(1+q) \sigma_{\vert \Sigma_\epsilon}+ {\rm Bal}((1+q)\sigma_{\vert \Sigma_\epsilon^c},\Sigma_\epsilon)}(x) +F_Q-F_\epsilon\ {\rm on} \ \Sigma_\epsilon \]
with equality on $\Sigma$. Reducing the inequality to potentials in the complex plane using \eqref{Sphere-to-Complex} for a stereographical projection about properly chosen $a$ and eliminating the $\log |z-a|$ term because of the normalization $\mu_Q(\mathbb{S}^2)=\eta_\epsilon(\mathbb{S}^2)=1$, we can apply the de la Vall\'{e}e Poussin theorem \cite[Theorem IV.4.5]{SaffTotik} for the image-measures in the complex plane and transfer the inequalities to the pre-images on the sphere and conclude
\[ \left(\mu_Q+{\rm Bal}(\mu_n,\Sigma_\epsilon)\right)_{\vert_\Sigma} \leq \left( (1+q) \sigma_{\vert \Sigma_\epsilon}+ {\rm Bal}((1+q)\sigma_{\vert \Sigma_\epsilon^c},\Sigma_\epsilon)\right)_{\vert_\Sigma}.\]
As the balayage measures are supported on the boundary of $\Sigma_\epsilon$, this is equivalent to $\mu_Q \leq (1+q) \sigma_{\vert \Sigma}$, which implies $\nu_j \leq (1+q)\sigma_{\vert \partial C_j}$. In the Remark at the end of Section 5 we shall see that $\sigma_{\vert \partial C_j}\equiv 0$ and hence 
\eqref{BalCond} follows. 

To complete the theorem, we derive \eqref{mu_Q^t} and \eqref{mu_Q^tcond} similarly, using \eqref{SphVarIneq} for $t_j$ instead.
 \end{proof}
 \vskip 2mm
 
 {\bf Remark:} We note that the material in this section can be generalized for Riesz $d-2$-potential interactions on $\mathbb{S}^d$. A careful analysis of the mass loss occurring after Riesz $(d-2)$-balayage is essential and will be pursued in a subsequent work.
 
\section{Equilibrium Support via Mergelyan}

\label{Sect4}

We now focus on describing the projection of the equilibrium support, $\Sigma^*$, which recall is described by (\ref{Frostman}). The hands-on approach of this section will require $a$ $priori$ knowledge of regularity of $\Sigma^*$, but we will see that this is not unwarranted in view of the next section of the article.

As above, let $n$ point charges of intensities $q_1, \cdots, q_n$ be placed at points $a_1, \cdots, a_n$ on $\mathbb{S}^2$. We assume that the equilibrium support $\Sigma$ is the complement in $\mathbb{S}^2$ of a $C^\infty$ smooth relatively open set $\Sigma^c \subset \mathbb{S}^2$. Assume further that $\Sigma^c$ has finitely many components, each of which is finitely-connected. 

Let the connected components of $\Sigma^c$ be named $C_1, \cdots, C_m$, and let $\Sigma^*$, $(\Sigma^c)^*$, $C_j^*$ denote stereographic projections to the plane. The projections of the $a_j$ will be called $z_j$. We assume as well that $\Sigma$ contains an interior point, and that the stereographic projection is taken from such a point.

The following theorem states that the equilibrium support $\Sigma$ is the stereographic preimage of the complement of a union of planar quadrature domains, the sum of whose orders is equal to the number of point charges.

\begin{thm} With everything set up as just described, each component $C_j^*$ of $(\Sigma^c)^*, j=1,\cdots m$, is a bounded quadrature domain in the plane. The sum of the orders of all the quadrature domains $C_j^*$ is $n$, the total number of point charges comprising the external field.
\end{thm}

\begin{proof}
Our strategy is to rewrite (\ref{Frostman}) in order to exploit Green's Theorem and get integrals on the boundary. Then, we will extract an integration formula for rational functions, which by means of Mergelyan's Theorem will be extended to holomorphic functions. Finally, using the orthogonal decomposition of the Hardy Space on smooth bounded domains, we will demonstrate that the $C_j^*$ are quadrature domains by proving their boundaries have meromorphic Schwarz functions.
 
 From (\ref{Frostman}), differentiate in $z$ on both sides, use $dw \wedge d\bar{w}=-2 i \cdot dA_w,$ and rearrange to obtain

\begin{equation}
\label{F2}
\frac{-1}{2 \pi i} \int_{\Sigma^*} \frac{1}{(1+|w|^2)^2} \frac{1}{w-z} dw \wedge d\bar{w} -\frac{1}{q+1} \sum_{i=1}^n \frac{1}{z-z_i} + \frac{\bar{z}}{1+z \bar{z}} \equiv 0,
\end{equation}

valid for all $z\in \Sigma^*$.

In light of the previous Theorem \ref{bal}, we can examine just one component $C_j^*$ at a time. So let $I_{C_j^*}$ denote the set of all indices such that $z_i \in C_j^*$ exactly when $i \in I_{C_j^*}$. Then consider the equilibrium problem \eqref{t_eq} on the charged sphere with total charge $t=1+q-\sum_{i\in C_j^*}q_i$, and with external field exerted by charges $q_i$ at the points $z_i$, for $i \in I_{C_j^*}$.

For convenience, let the various components of the support of the resulting equilibrium measure be called $S_0, S_1, S_2, \cdots, S_K,$ where $S_0$ is the unbounded component. (Recall that we have projected from an interior point of $\Sigma$, so in the plane all points near $\infty$ belong to $\Sigma^*.$) We will also use $\gamma_j$ to denote the outer boundary curve of $C_j^*.$ We begin by examining what happens for $z$ in the unbounded component $S_0$.

CASE 1: $z \in S_0$.

 Noting that
 \[ \frac{1}{(1+|w|^2)^2}= \frac{\partial}{\partial \bar{w}} \frac{\bar{w}}{1+w \bar{w}},\]
  we modify (\ref{F2}) as follows, recalling that our external field is now considered only as comprising the charges at the $z_i \in C_j^*$:

\[-\frac{1}{2\pi i} \sum_{k=0}^K \int_{S_k}\frac{\partial}{\partial\bar{w}} \big{(} \frac{\bar{w}}{1+w \bar{w}}\big{)} \frac{1}{w-z}dw \wedge d\bar{w} + \frac{\bar{z}}{1+z\bar{z}} = \frac{1}{q+1} \sum_{i\in I_{C_j^*}} \frac{q_i}{z-z_i}.    \]

Next, we use Stokes's Theorem and the $C^\infty$ Cauchy Integral Formula (see e.g. \cite{BellBook}) to evaluate the area integrals on the left side.

Let $R$ be arbitrarily large, $|z|<R$, and let $D_R$ be the disc centered at the origin of radius $R$. For the integral over the unbounded component, break the integral into two pieces: one inside and one outside $D_R$, and use the $C^\infty$ Cauchy Formula on the inside portion
 \[-\frac{1}{2\pi i}\int_{S_0}\frac{\partial}{\partial\bar{w}}\big{(} \frac{\bar{w}}{1+w \bar{w}}\big{)} \frac{1}{w-z}dw\wedge d\bar{w}=\]
 
 \[-\frac{1}{2\pi i }\int_{S_0 \cap D_R}\frac{\partial}{\partial\bar{w}}\big{(} \frac{\bar{w}}{1+w \bar{w}}\big{)} \frac{1}{w-z}dw\wedge d\bar{w} - \frac{1}{2 \pi i} \int_{D_R^c}\frac{\partial}{\partial\bar{w}}\big{(} \frac{\bar{w}}{1+w \bar{w}}\big{)} \frac{1}{w-z}dw\wedge d\bar{w}= \]
 
 \[-\frac{\bar{z}}{1+z\bar{z}} + \frac{1}{2 \pi i}\int_{\partial D_R} \frac{1}{w+\frac{1}{\bar{w}}} \frac{1}{w-z} dw - \frac{1}{2 \pi i} \int_{\gamma_j} \frac{1}{w+\frac{1}{\bar{w}}} \frac{1}{w-z}dw-\]
 
 \[ \frac{1}{2 \pi i} \int_{D_R^c} \frac{1}{(1+|w|^2)^2} \frac{1}{w-z}dw\wedge d\bar{w}.\]
 
  In this equality, let $R \to \infty$. Then the area integral over $D_R^c$ is on the order 
  \[\frac{1}{R^4}\cdot \frac{1}{R} \cdot R^2\] 
  and so tends to $0$. The boundary integral over $\partial D_R$ goes on the order 
  \[ \frac{R}{R^2} \cdot \frac{1}{R} \cdot R ,\]
   and so also tends to $0$. We conclude that
   
    \[-\frac{1}{2\pi i}\int_{S_0}\frac{\partial}{\partial\bar{w}}\big{(} \frac{\bar{w}}{1+w \bar{w}}\big{)} \frac{1}{w-z}dw\wedge d\bar{w}= -\frac{\bar{z}}{1+z\bar{z}}  - \frac{1}{2 \pi i} \int_{\gamma_j} \frac{1}{w+\frac{1}{\bar{w}}} \frac{1}{w-z}dw.\] 
    This takes care of the unbounded portion of the integration.

Next, we analyze the integration over the bounded $S_k$, $k>0$. Here we can use Stokes's Theorem straight away, without recourse to the Cauchy Formula. This is because, since $z$ is outside $S_k$, the integrand 

\[ \frac{1}{(1+|w|^2)^2} \cdot \frac{1}{w-z}\]

 is smooth up to the boundary, and $1/(w-z)$ is holomorphic. Notice then that
 
  \[-\frac{1}{(1+|w|^2)^2} \frac{1}{w-z}dw \wedge d\bar{w} = d(\frac{1}{w+\frac{1}{\bar{w}}} \frac{1}{w-z} dw).\]
  
   By Stokes's Theorem, we conclude that
   
    \[-\frac{1}{2\pi i} \int_{S_k} \frac{1}{(1+|w|^2)^2} \frac{1}{w-z} dw \wedge d\bar{w}= \frac{1}{2 \pi i}\int_{\partial S_k} \frac{1}{w+\frac{1}{\bar{w}}} \frac{1}{w-z} dw.\]

The above computations have found equivalent boundary versions of the various area integrals: take them all and substitute into (\ref{F2}). The $\frac{\bar{z}}{1+z\bar{z}}$ term cancels. Observe the boundary integrals are occurring over the boundaries of all the bounded components $S_1, \cdots, S_K$, with standard orientation. So we can write them as occurring over the boundary of the complement $C_j^*$ in reverse. By introducing a factor of $-1$ we get 

\[ \frac{1}{2 \pi i} \int_{\partial C_j^*} \frac{1}{w+\frac{1}{\bar{w}}} \frac{1}{w-z}dw= \frac{1}{q+1}\sum_{i \in I_{C_j^*}} \frac{q_i}{z_i-z}. \]

We'll keep this formula in mind and turn attention to the case when $z$ is located in a bounded component.

CASE 2: $z$ is in a bounded component of $\Omega$.

In case $z \in S_k,$ $k>0$, we can manipulate (\ref{F2}) in much the same way as in CASE 1.  Use the $C^\infty$ Cauchy Formula on the component $S_k$, and on all other bounded components use Stokes's Theorem. On the unbounded component, first break the area integral into portions inside and outside a large disc $D_R$. On the inner part use Stokes's Theorem. Then let $R \to \infty$, and find that the area and boundary integrals involving $D_R$ vanish, leaving only an integration over the outer boundary curve $\gamma_j$.  And now again all the area integration has been moved to the boundary. The orientations align themselves in such a way that after substituting into (\ref{F2}), the same final formula occurs as in CASE 1. 

So we conclude that for any $z \in \Sigma^*$, the following formula is valid:

\begin{equation}
\label{F3}
 \frac{1}{2 \pi i} \int_{\partial C_j^*} \frac{1}{w+\frac{1}{\bar{w}}} \frac{1}{w-z}dw= \frac{1}{q+1}\sum_{i \in I_{C_j^*}} \frac{q_i}{z_i-z}.
 \end{equation}

We are now ready to see how this formula leads to a quadrature rule for rational functions.  Differentiate our new equation (\ref{F3}) any number of times in $z$, and for any positive integer $r$, 

\[\frac{1}{2 \pi i}  \int_{\partial C_j^*} \frac{1}{w+\frac{1}{\bar{w}}} \frac{1}{(w-z)^r}dw= \frac{1}{q+1}\sum_{i \in I_{C_j^*}} \frac{q_i}{(z_i-z)^r}.\]

By linearity and the Fundamental Theorem of Algebra, this means that for any rational function with poles only in $\Sigma^*$, 

\[\frac{1}{2 \pi i} \int_{\partial C_j^*} \frac{1}{w+\frac{1}{\bar{w}}} R(w)dw= \frac{1}{q+1}\sum_{i \in I_{C_j^*}} q_i R(z_i).\]

Now we use Mergelyan's Theorem.  Let $h \in A^\infty(C_j^*)$; that is, $h$ is analytic and smooth up to the boundary. We can, using only rational functions with poles in $\Sigma^*$, uniformly approximate the function $h$. By uniform convergence, this yields:

\[\frac{1}{2\pi i} \int_{\partial C_j^*} \frac{1}{w+\frac{1}{\bar{w}}} h(w)dw=\frac{1}{q+1} \sum_{i \in I_{C_j^*}} q_ih(z_i). \]

After this, rewrite the right hand side as a sum of Cauchy integrals, and subtract them to the left side. The result is that, for any $h \in A^\infty(C_{j}^*)$, 

\[\frac{1}{2\pi i} \int_{\partial C_{j}^*}\big{[}\frac{1}{w+\frac{1}{\bar{w}}}-\sum_{i \in I_{C_j^*}} \frac{q_i}{q+1} \frac{1}{w-z_i}\big{]}h(w)dw=0. \]
But $A^\infty$ is dense in the Hardy Space of $C_j^*$. That means the bracketed part of the integrand is orthogonal to the Hardy Space, and by the orthogonal decomposition of the Hardy Space (e.g. \cite{BellBook}), this means that there exists a function $H$, holomorphic and smooth up to the boundary of $C_{j}^*$, such that for all $w$ on the boundary of $C_{j}^*$,

 \[ \frac{1}{w+\frac{1}{\bar{w}}}-\sum_{i \in I_{C_j^*}} \frac{q_i}{q+1} \frac{1}{w-z_i}=H(w).\]

We are now in essence finished, because we can simply solve for $\bar{w}$ in this equation to see that $\bar{w}$ has the boundary values of a meromorphic function. This means that $C_{j}^*$ has a meromorphic Schwarz function, and consequently is a quadrature domain. 

Counting poles with the argument principle, we will see that the order of $C_j^*$ as a quadrature domain is the cardinality of $I_{C_j^*}$. Let us use the argument principle on the boundary cycle of $C_j^*$. Let $\nu(\cdot)$ be the winding number of a function around the boundary cycle, and let $\zeta(\cdot)$ be the number of zeroes of a function occurring in $C_{j}^*$, and let $p(\cdot)$ denote the number of poles occurring in $C_{j}^*.$ Then the argument principle ensures that for meromorphic functions smooth up to the boundary without roots or poles on the boundary, $\nu(\cdot)=\zeta(\cdot)-p(\cdot).$

Consider the function 
\[\frac{1}{w+S(w)^{-1}},\]
where $S(w)$ is the Schwarz function of the boundary of $C_{j}^*$.  We have just seen that along the boundary of $C_{j}^*$, \[\frac{1}{w+\frac{1}{S(w)}}=H(w)-\sum_{i\in I_{C_j^*}}\frac{q_i}{q+1} \frac{1}{w-z_i}.\]
 The number of poles on the right is exactly $card(I_{C_j^*})$, and so this is also 
 \[p(\frac{1}{w+\frac{1}{S(w)}}).\]
 These poles are by inspection the roots of $w+S(w)^{-1}$. So \[\zeta(w+\frac{1}{S(w)})=card(I_{C_j^*}).\]
   We can also determine that the winding number is \[\nu(w+S(w)^{-1})=\nu(w+\frac{1}{\bar{w}})=\nu(\frac{1+|w|^2}{\bar{w}}).\]
    The numerator is real-valued and makes no contribution to the winding number, so this further simplifies to \[\nu(\frac{1}{\bar{w}})=\nu(\frac{w}{|w|^2})=\nu(w)\]
     (again, the real valued $|w|^2$ has made no contribution). By the argument principle, we now have:

\[card(I_{C_j^*})-p(w+S(w)^{-1})=\nu(w). \] 

Notice now that $w+S(w)^{-1}$ has poles exactly at the roots of $S(w)$.  In other words, $p(w+\frac{1}{S(w)})=\zeta(S(w)).$ Substituting into our formula now yields: $card(I_{C_j^*})-\zeta(S(w))=\nu(w),$ and deploying the argument principle on $\zeta$ gives \[card(I_{C_j^*})-(\nu(S(w)))+p(S(w))=\nu(w).\]
 We are nearly finished. Looking at $S$, we see that \[\nu(S(w))=\nu(\bar{w}).\]
  Since $Arg(\bar{w})=-Arg(w)$ for all $w$, we see that $\nu(S(w))=-\nu(w).$ Plugging in one last time,

\[card(I_{C_j^*})+\nu(w)-p(S(w))=\nu(w), \]

whence

\[p(S(w))=card(I_{C_j^*}).\]

And now, since the poles of the Schwarz function count the nodes of evaluation in the quadrature identity, we conclude that $C_{j}^*$ is a quadrature domain of order $card(I_{C_j^*}).$
\end{proof}

We remark here that every $C_j^*$ must include one of the $z_i$, since otherwise $p(S(w))=0$, meaning that $\bar{w}$ is holomorphic on $C_j^*$, which is untenable. That means the $z_i$ are all contained in some of the $C_j^*$, and each $z_i$ can be a member of at most one $C_j^*$ since the $C_j^*$ are distinct connected components.  Thus $\sum_j card(I_j^*)=n$, and we are finished. The above argument principle approach did implicitly assume that $w=0$ is not on the boundary of any of the $C_j^*$, but this can be effected by rotating the sphere to slightly alter the north pole.

\section{An alternate approach}
\label{Sect5}

In this section, we gain the same description of the equilibrium support, from a point of view of ideas from Aharanov and Shapiro \cite{AS}, and Hedenmalm and Makarov \cite{HM}.  This approach will demonstrate an algebraic boundary for the equilibrium support, and give the quadrature property via a Schwarz function (recall that in the previous section we needed regularity of the boundary). But whereas in Section \ref{Sect4} we argued from Frostman's condition using the area integral on $\Sigma^*$, in this section we exploit the symmetry of $\mathbb{S}^2$ to pass the area integral in Frostman's condition directly to the complement $(\Sigma^c)^*$.

As before, let $\Sigma \subset \mathbb{S}^2$ be the support of the equilibrium support $\mu_Q$ in the presence of the point charges $q_j$ present at points $a_j \in \mathbb{S}^2.$  By an equatorial stereographic projection to the plane, and invoking Frostman's condition, we begin again at:

\begin{equation}
\label{Frostman2}
\begin{split}
\frac{1+q}{\pi} \int_{\Sigma^*} \frac{1}{(1+|w|^2)^2}\ln|w-z|^{-1}dA_w + (1+q)\ln\sqrt{1+|z|^2} \\ +\sum_{j=1}^n q_j \ln|z-z_j|^{-1}=const,
\end{split}
\end{equation}
where $q$ is the sum of all charge intensities, $\Sigma^*$ is the projection of $\Sigma$ to the plane, and $z_j$ is the projection of $a_j$ to the plane, and the equation holds for $z \in \Sigma^*$.

Importantly, Frostman's condition may be written in this way regardless of the boundary of $\Sigma^*$, as described in \cite{HM}, where the authors explain that the equilibrium measure's density is the Laplacian of the external field, throughout the support.  

On the other hand, by symmetry, the potential exerted by the uniform measure $(1+q)\sigma_2$ on $\mathbb{S}^2$ is constant over the whole of $\mathbb{S}^2$. After expressing this logarithmic potential in the plane via projection, we conclude that for all $z \in \mathbb{C}$, 

\[\frac{1+q}{\pi} \int_\mathbb{C} \frac{1}{(1+|w|^2)^2}\ln|w-z|^{-1}dA_w + (1+q)\ln\sqrt{1+|z|^2}=const, \]

a possibly different constant than the one above.

Upon splitting the integral over the whole plane into $\int_{\Sigma^*}+\int_{(\Sigma^*)^c}$, and combining with (\ref{Frostman}), this gives:

\[-\frac{1+q}{\pi} \int_{(\Sigma^*)^c} \frac{1}{(1+|w|^2)^2}\ln|w-z|^{-1}dA_w = \sum_{j=1}^n q_j \ln|z-z_j|+const, \]

where the constant has changed yet again.

At this point, differentiate each side in $z$, to obtain

\begin{equation}
\label{wqd}
\int_{(\Sigma^*)^c} \frac{1}{(1+|w|^2)^2} \cdot \frac{1}{w-z}dA_w=\sum_{j=1}^n \frac{\pi q_j}{1+q}\cdot \frac{1}{z_j-z}.
\end{equation}

This already suggests that $(\Sigma^*)^c$ is a quadrature domain with respect to weighted Lebesgue measure, but as we did before, we can further conclude that $(\Sigma^*)^c$ is a quadrature domain with respect to unweighted Lebesgue measure. 

In fact, we can employ the argument used by Aharanov and Shapiro when they connected the Schwarz function to quadrature identities \cite{AS}, suitably modified to our current situation. Following their approach, consider the function $u=\frac{1}{(1+|w|^2)^2}\chi$, where $\chi$ is the indicator function for $(\Sigma^*)^c.$ Letting $I(z)$ denote the integral on the left hand side of (\ref{wqd}), note that $I(z)$ is the Cauchy transform of $u$. 

By Lemma 2.1 of \cite{AS}, $I(z)$ is continuous on all of $\mathbb{C}$, and in the distributional sense we have:

\[\frac{\partial I}{\partial \bar{z}} = -\pi u(z).   \]

Still following the flow of \cite{AS}, let

\[g(z)=I(z)+\frac{\pi \bar{z}}{1+|z|^2}. \]

Applying $\frac{\partial}{\partial \bar{z}}$ in the distributional sense on $(\Sigma^*)^c$, we see that $g(z)$ is `weakly' holomorphic there. But by Weyl's Lemma, that means $g$ is legitimately analytic in $(\Sigma^*)^c$. Note also that $g$ is continuous up to the boundary of $(\Sigma^*)^c.$

But (\ref{wqd}) says also that $I(z)$ coincides on $\Sigma^*$ with a rational function which has exactly $n$ simple poles, all inside $(\Sigma^*)^c$.  Since the right side of (\ref{wqd}) and $I$ are continuous, we conclude that $(\ref{wqd})$ holds even on the boundary $\partial (\Sigma^*)^c$. So let $R(z)$ denote the rational function on the right hand side of (\ref{wqd}), and our conclusion is that the function

\[g(z)-R(z) \]
is meromorphic on $(\Sigma^*)^c$, is continuous to the boundary, and we have on the boundary $\partial (\Sigma^*)^c$, the equality:

\[\frac{\bar{z}}{1+|z|^2}=\frac{1}{\frac{1}{\bar{z}}+z}= g(z)-R(z).\]

Solving this equation for $\bar{z}$, we see that $ (\Sigma^*)^c$ has a meromorphic Schwarz function, and thus $(\Sigma^*)^c$ is a quadrature domain.  From here, we may count poles with the argument principle to conclude that it has order $n$, and we are guaranteed that the boundary is algebraic.

\textbf{Remark.} A consequence of this is that the Lebesgue measure of the boundary of the equilibrium support is $0$, as referenced in the proof of Theorem \ref{bal}. To reiterate, if we imagine letting the charge intensities grow, the components do not interact with each other until their boundaries touch, after which point they merge into a larger component.

\section{Examples}
\label{Sect6}

In this final section we present two examples of two-point configurations, one symmetric and the other asymmetric. In our examples, we place two point charges on the sphere and use the fact that simply connected quadrature domains are rational images of the unit disc. Via the Bergman kernel function, the corresponding quadrature nodes and coefficients in the quadrature identity can be determined. The Schwarz function can also be used to find the quadrature data with respect to Lebesgue and spherical measures.

For more than two point charges, we remark that multiply connected quadrature domains can arise, and in this case mapping conformally from the unit disc is no longer possible. For example, consider three or more point charges equidistributed on a circle, whose charge intensities are equal and growing. The resulting quadrature domain begins as disjoint discs, then coalesces into a single doubly-connected domain, and finally the hole closes leaving a simply connected domain. This case is studied for instance in \cite{CrowdyMultipolar} and \cite{GuSpecial}. More complicated configurations can be studied as well \cite{CM}.  For more on topology of quadrature domains, see \cite{LM}.  

\subsection{Two symmetric charges}

In the case of two symmetric charges whose caps of influence overlap, place the north pole of the sphere inside the equilibrium support in such a way that the point charges are symmetrically placed about the south pole, along the real line of the Riemann sphere. In this case, the region of charge exclusion projects to become a symmetric two-point quadrature domain in the plane, with quadrature nodes along the real axis symmetrically placed about the origin. (As mentioned in the introduction, this configuration was recently studied in \cite{CdRK} by other methods.)

Such a quadrature domain will be the conformal image of the unit disc under a map of the form $\varphi: z \to \frac{A}{C-z}+\frac{A}{-C-z}$, with $A,C$ real parameters.  Once the mapping parameters $A,C$ are given, one can calculate the nodes and coefficients defining the quadrature domain using the Bergman kernel.  In fact the quadrature nodes occur at $\pm \varphi(1/C)$, and the coefficients in the quadrature identity are both $\frac{\pi A}{C^2}\varphi'(1/C)$. 

The poles and residues of the Schwarz function $S(z)$ determine the quadrature data of a quadrature domain with respect to planar Lebesgue measure. With respect to spherical measure $\frac{1}{\pi(1+|z|^2)^2}dA$, the quadrature data are instead determined by the so-called `Spherical Schwarz Function', $\tilde{S}(z)=\frac{S(z)}{\pi(1+zS(z))}.$ By the use of the spherical Schwarz function, it is understood that spherical quadrature domains and planar quadrature domains are related via stereographic projection, although the quadrature data differ in each measure (cf. \cite{GT}). From the standpoint of fluid dynamics, this type of approach has been used, for instance in \cite{CC}.  

Now, the Schwarz function of our Neumann oval can be written as $\varphi(\frac{1}{\varphi^{-1}(z)})$. This comes from symmetry about the real line, together with the fact that the unit disc has Schwarz function $z^{-1}.$  Utilizing this formula in $\tilde{S}$, one can explicitly calculate quadrature data for $\varphi(\mathbb{D})$ with respect to the spherical measure in terms of the mapping parameters $A,C$.  We do not list these formulas here, but note that they are algebraic in $A,C.$

The result of all this is that we can choose the mapping parameters $A,C$, then compute exactly the planar and spherical quadrature data of the resulting Neumann oval. After stereographic pre-projection to the sphere, the spherical measure's quadrature data of course give the location and intensities of the point charges giving rise to the Neumann oval.

\begin{figure}[ht]
\centering
\includegraphics[scale=.5]{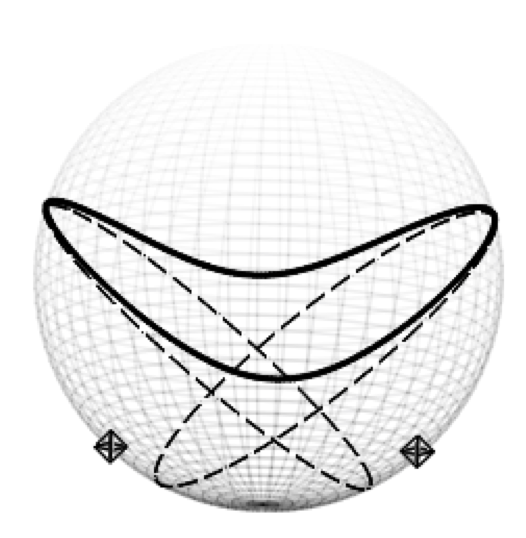}
\includegraphics[scale=.45]{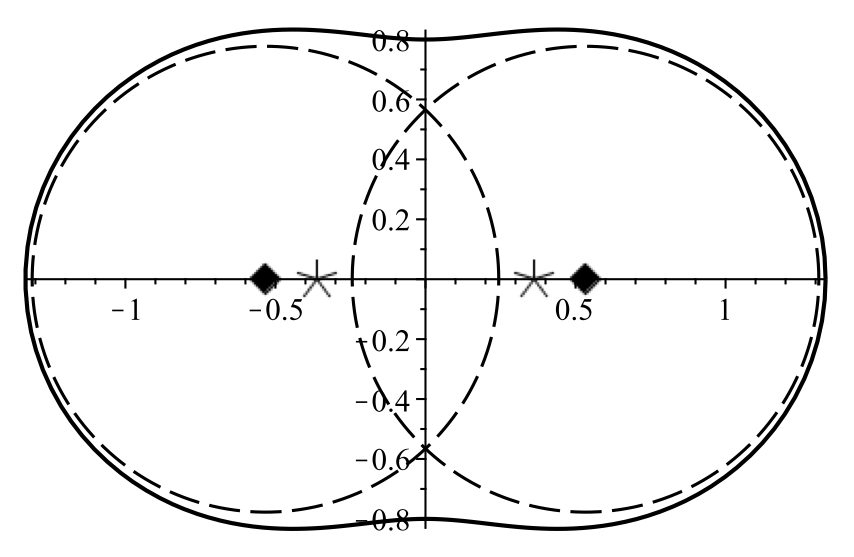}
\caption{Two charges - symmetric case}
\label{symetric}
\end{figure}

We implement this for a particular case, and plot the results in Maple.  Place point charges of intensity $q=\frac{41-3\sqrt{41}}{82}$ at the points $(\pm \frac{16}{25},0,-\frac{3\sqrt{41}}{25})$ of the unit sphere. Figure 1 shows the boundary of the resulting equilibrium support, together with the individual caps of influence.   After projection, the region of charge exclusion is a Neumann oval, being a quadrature domain with respect to both spherical and Lebesgue measure. With respect to Lebesgue measure, it has quadrature nodes at the points $\pm \frac{8}{15}$, with quadrature-identity coefficients $\frac{136 \pi}{225}$.  It is the image of the unit disc under the map $\varphi=\frac{2}{2-z}+\frac{2}{-2-z}.$ Figure $2$ presents the projection: the diamonds are the Lebesgue nodes, the asterisks are the spherical nodes, and the dotted circles are the discs of area $\frac{136 \pi}{225}$ about the Lebesgue nodes. We remark that the spherical nodes are closer to the origin of the plane than the Lebesgue nodes, since the spherical measure counts area further from the origin less.

\begin{figure}[ht]
	\centering
	\includegraphics[scale=.5]{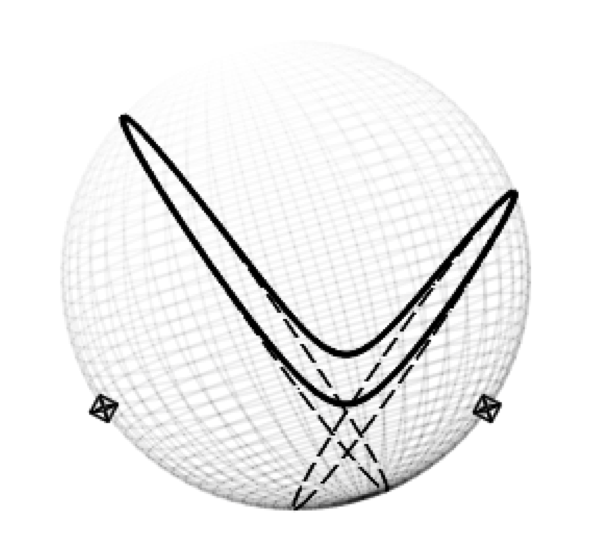}
	\includegraphics[scale=.45]{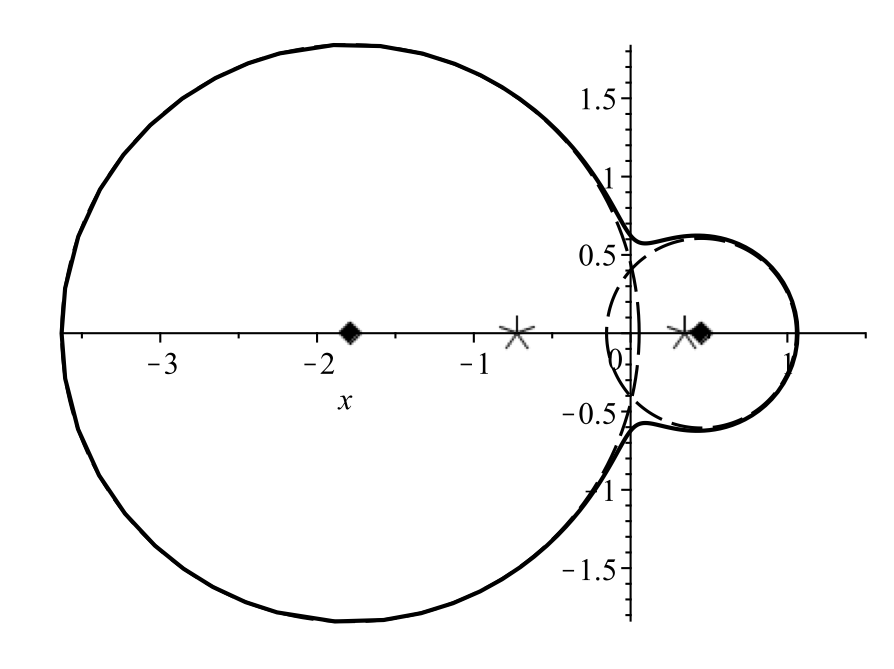}
	\caption{Two charges - asymmetric case}
	\label{asymetric}
\end{figure}

\subsection{Two asymmetric charges}
The case of two asymmetric charges on the sphere can be handled similarly. Consider the image of the conformal map from the unit disc given by $\varphi=\frac{1}{0.8-z}+\frac{1.7}{-1.2-z}.$  We compute $S$ and $\tilde{S}$, and find their poles and residues via Maple numerically (and we round to two decimal places). The conclusion is that we have a quadrature domain with respect to Lebesgue measure with nodes at approximately $-1.79$ and $0.45$, and the corresponding coefficients in the quadrature identity are approximately $10.66$ and $1.15$ respectively.  By computing the quadrature data with respect to spherical measure, we find the domain arises as the projection of the region of charge exclusion on the sphere, with charges placed at approximately $(-0.95,0,-0.31)$ and $(0.62,0,-0.79)$ with respective intensities approximately $q_1=0.12$, $q_2=0.07$. Figure $3$ displays the boundary of the spherical equilibrium support with the individual caps of influence, and Figure $4$ displays the stereographic projection, with the same conventions as in the previous example.

\noindent{\bf Acknowledgment.} The authors would like to thank the Institute for Computational and Experimental Research in Mathematics in Providence, RI, for their hospitality, where part of this work was initiated.  Additionally the authors would like to extend thanks to Edward Saff, Bjorn Gustafsson, and Darren Crowdy for valuable discussions and insights.

	\end{document}